\documentclass[a4paper,11pt]{amsart}
\usepackage{mathrsfs}
\usepackage{cases}
\usepackage{amsfonts}
\usepackage{graphicx}
\usepackage{amsmath}
\usepackage{amssymb}
\usepackage[all]{xypic}
\usepackage[all]{xy}
\usepackage{color}
\usepackage{colordvi}
\usepackage{multicol}
\begin{document}
\input xy
\xyoption{all}

\newcommand{\dd}{\operatorname{dom.dim}\nolimits}
\newcommand{\Aus}{\operatorname{Aus}\nolimits}
\renewcommand{\mod}{\operatorname{mod}\nolimits}
\newcommand{\proj}{\operatorname{proj}\nolimits}
\newcommand{\inj}{\operatorname{inj}\nolimits}
\newcommand{\rad}{\operatorname{rad}\nolimits}
\newcommand{\soc}{\operatorname{soc}\nolimits}
\newcommand{\Ginj}{\operatorname{Ginj}\nolimits}
\newcommand{\Mod}{\operatorname{Mod}\nolimits}
\newcommand{\R}{\operatorname{R}\nolimits}
\newcommand{\End}{\operatorname{End}\nolimits}
\newcommand{\ob}{\operatorname{Ob}\nolimits}
\newcommand{\Ht}{\operatorname{Ht}\nolimits}
\newcommand{\ind}{\operatorname{ind}\nolimits}
\newcommand{\rep}{\operatorname{rep}\nolimits}
\newcommand{\Ext}{\operatorname{Ext}\nolimits}
\newcommand{\Tor}{\operatorname{Tor}\nolimits}
\newcommand{\Hom}{\operatorname{Hom}\nolimits}
\newcommand{\Pic}{\operatorname{Pic}\nolimits}
\newcommand{\Coker}{\operatorname{Coker}\nolimits}
\newcommand{\res}{\operatorname{res}\nolimits}

\newcommand{\Gproj}{\operatorname{Gproj}\nolimits}
\newcommand{\irr}{\operatorname{irr}\nolimits}
\newcommand{\RHom}{\operatorname{RHom}\nolimits}
\renewcommand{\deg}{\operatorname{deg}\nolimits}
\renewcommand{\Im}{\operatorname{Im}\nolimits}
\newcommand{\Ker}{\operatorname{Ker}\nolimits}
\newcommand{\Aut}{\operatorname{Aut}\nolimits}
\newcommand{\Id}{\operatorname{Id}\nolimits}
\newcommand{\Qcoh}{\operatorname{Qch}\nolimits}
\newcommand{\CM}{\operatorname{CM}\nolimits}
\newcommand{\Cp}{\operatorname{Cp}\nolimits}
\newcommand{\coker}{\operatorname{Coker}\nolimits}
\renewcommand{\dim}{\operatorname{dim}\nolimits}
\renewcommand{\div}{\operatorname{div}\nolimits}
\newcommand{\Ab}{{\operatorname{Ab}\nolimits}}
\renewcommand{\Vec}{{\operatorname{Vec}\nolimits}}
\newcommand{\pd}{\operatorname{proj.dim}\nolimits}
\newcommand{\id}{\operatorname{inj.dim}\nolimits}
\newcommand{\Gd}{\operatorname{G.dim}\nolimits}
\newcommand{\gldim}{\operatorname{gl.dim}\nolimits}
\newcommand{\dimv}{\operatorname{\underline{dim}}\nolimits}
\newcommand{\sdim}{\operatorname{sdim}\nolimits}
\newcommand{\add}{\operatorname{add}\nolimits}
\newcommand{\pr}{\operatorname{pr}\nolimits}
\newcommand{\oR}{\operatorname{R}\nolimits}
\newcommand{\oL}{\operatorname{L}\nolimits}
\newcommand{\Perf}{{\mathfrak Perf}}
\newcommand{\cc}{{\mathcal C}}
\newcommand{\ce}{{\mathcal E}}
\newcommand{\cs}{{\mathcal S}}
\newcommand{\cf}{{\mathcal F}}
\newcommand{\cx}{{\mathcal X}}
\newcommand{\ct}{{\mathcal T}}
\newcommand{\cu}{{\mathcal U}}
\newcommand{\cv}{{\mathcal V}}
\newcommand{\cn}{{\mathcal N}}
\newcommand{\ch}{{\mathcal H}}
\newcommand{\ca}{{\mathcal A}}
\newcommand{\cb}{{\mathcal B}}
\newcommand{\ci}{{\mathcal I}}
\newcommand{\cj}{{\mathcal J}}
\newcommand{\cm}{{\mathcal M}}
\newcommand{\cp}{{\mathcal P}}
\newcommand{\cq}{{\mathcal Q}}
\newcommand{\cg}{{\mathcal G}}
\newcommand{\cw}{{\mathcal W}}
\newcommand{\co}{{\mathcal O}}
\newcommand{\cd}{{\mathcal D}}
\newcommand{\ck}{{\mathcal K}}
\newcommand{\calr}{{\mathcal R}}
\newcommand{\ol}{\overline}
\newcommand{\ul}{\underline}
\newcommand{\st}{[1]}
\newcommand{\ow}{\widetilde}
\renewcommand{\P}{\mathbf{P}}
\newcommand{\pic}{\operatorname{Pic}\nolimits}
\newcommand{\Spec}{\operatorname{Spec}\nolimits}
\newtheorem{theorem}{Theorem}[section]
\newtheorem{acknowledgement}[theorem]{Acknowledgement}
\newtheorem{algorithm}[theorem]{Algorithm}
\newtheorem{axiom}[theorem]{Axiom}
\newtheorem{case}[theorem]{Case}
\newtheorem{claim}[theorem]{Claim}
\newtheorem{conclusion}[theorem]{Conclusion}
\newtheorem{condition}[theorem]{Condition}
\newtheorem{conjecture}[theorem]{Conjecture}
\newtheorem{construction}[theorem]{Construction}
\newtheorem{corollary}[theorem]{Corollary}
\newtheorem{criterion}[theorem]{Criterion}
\newtheorem{definition}[theorem]{Definition}
\newtheorem{example}[theorem]{Example}
\newtheorem{exercise}[theorem]{Exercise}
\newtheorem{lemma}[theorem]{Lemma}
\newtheorem{notation}[theorem]{Notation}
\newtheorem{problem}[theorem]{Problem}
\newtheorem{proposition}[theorem]{Proposition}
\newtheorem{remark}[theorem]{Remark}
\newtheorem{solution}[theorem]{Solution}
\newtheorem{summary}[theorem]{Summary}
\newtheorem*{thm}{Theorem}

\def \bp{{\mathbf p}}
\def \bA{{\mathbf A}}
\def \bL{{\mathbf L}}
\def \bF{{\mathbf F}}
\def \bS{{\mathbf S}}
\def \bC{{\mathbf C}}

\def \Z{{\Bbb Z}}
\def \F{{\Bbb F}}
\def \C{{\Bbb C}}
\def \N{{\Bbb N}}
\def \Q{{\Bbb Q}}
\def \G{{\Bbb G}}
\def \P{{\Bbb P}}
\def \K{{\Bbb K}}
\def \E{{\Bbb E}}
\def \A{{\Bbb A}}
\def \BH{{\Bbb H}}
\def \T{{\Bbb T}}

\title[Cohen-Macaulay Auslander algebras]{Cohen-Macaulay Auslander algebras of gentle algebras}
\author[Chen]{Xinhong Chen}
\address{Department of Mathematics, Southwest Jiaotong University, Chengdu 610031, P.R.China}
\email{chenxinhong@swjtu.edu.cn}

\author[Lu]{Ming Lu$^\dag$}
\address{Department of Mathematics, Sichuan University, Chengdu 610064, P.R.China}
\email{luming@scu.edu.cn}
\thanks{$^\dag$ Corresponding author}


\begin{abstract}
For any gentle algebra $\Lambda=KQ/\langle I\rangle$, following Kalck, we describe the quiver and the relations for its Cohen-Macaulay Auslander algebra $\Aus(\Gproj\Lambda)$ explicitly, and obtain some properties, such as $\Lambda$ is representation-finite if and only if $\Aus(\Gproj\Lambda)$ is; if $Q$ has no loop and any indecomposable $\Lambda$-module is uniquely determined by its dimension vector, then any indecomposable $\Aus(\Gproj\Lambda)$-module is uniquely determined by its dimension vector.
\end{abstract}

\maketitle

\section{Introduction}

The concept of Gorenstein projective modules over any ring can be dated back to \cite{AB}, where Auslander and Bridger introduced the modules of $G$-dimension zero over a Noetherian rings, and is formed by Enochs and Jenda \cite{EJ}. This class of modules satisfies some good stable properties, becomes a main ingredient in the relative homological algebra, and is widely used in the representation theory of algebras and algebraic geometry, see e.g. \cite{AB,AR2,EJ,Bu,Ha1,Be}. It also plays as an important tool to study the representation theory of Gorenstein algebra, see e.g. \cite{AR2,Bu,Ha1}.

Gorenstein algebra $\Lambda$, where by definition $\Lambda$ has finite injective dimension both as a left and a right $\Lambda$-module, is inspired from commutative ring theory.
A fundamental result of Buchweitz \cite{Bu} and Happel \cite{Ha1} states that for a Gorenstein algebra $\Lambda$, its singularity category is triangle equivalent to the stable category of Gorenstein projective (also called (maximal) Cohen-Macaulay) $\Lambda$-modules, which generalizes Rickard's result \cite{Ri} on self-injective algebras.

For any Artin algebra $\Lambda$, denote by $\Gproj\Lambda$ its subcategory of Gorenstein projective modules. If $\Gproj\Lambda$ has only finitely many isomorphism classes of indecomposable objects, then $\Lambda$ is called CM-\emph{finite}. In this case, inspired by the definition of Auslander algebra, the Cohen-Macaulay Auslander algebra (also called the relative Auslander algebra) is defined to be $\End_\Lambda(\bigoplus_{i=1}^n E_i)^{op}$, where $E_1,\dots,E_n$ are all pairwise non-isomorphic indecomposable Gorenstein projective modules \cite{Be1,Be,LZ}.
A CM-finite algebra $\Lambda$ is Gorenstein if and only if $\gldim \Aus(\Gproj\Lambda)<\infty$ \cite{LZ,Be}.
Furthermore, for any two Gorenstein Artin algebras $A$ and $B$ which are CM-finite, if $A$ and $B$ are derived equivalent, then their Cohen-Macaulay Auslander algebras are also
derived equivalent \cite{P}.

As an important class of Gorenstein algebras \cite{GR}, gentle algebras were introduced in \cite{AS} as appropriate context for the investigation of algebras derived equivalent to hereditary algebras of type $\tilde{\A}_n$. Many important algebras are gentle, such as tilted algebras of type $\A_n$, algebras derived equivalent to $\A_n$-configurations of projective lines \cite{Bur} and also the cluster-tilted algebras of type $\A_n$ \cite{BMR}, and type $\tilde{\A}_n$ \cite{ABCP}.
It is interesting to notice that the class of gentle algebras is closed under derived equivalence \cite{SZ}. Recently, Kalck \cite{Ka} proves that the singularity category of an arbitrary gentle algebra is a finite product of $n$-cluster categories of type $\A_1$. From \cite{Ka}, it is easy to see that gentle algebras are CM-finite, which inspires us to study the properties of their Cohen-Macaulay Auslander algebras.

In this paper, our aim is to study the Cohen-Macaulay Auslander algebras of gentle algebras. Let $\Lambda=KQ/\langle I\rangle$ be a gentle algebra. First, we explicitly describe the quiver and relations of $\Aus(\Gproj\Lambda)=KQ^{Aus}/\langle I^{Aus}\rangle$, see Theorem \ref{main theorem 1}. Second, we prove that $\Lambda$ is representation-finite if and only if $\Aus(\Gproj\Lambda)$ is, see Theorem \ref{main theorem 2}. Third, if $Q$ has no loop, and any indecomposable $\Lambda$-module $M$ is uniquely determined by its dimension vector, then any indecomposable $\Aus(\Gproj\Lambda)$-module $N$ is uniquely determined by its dimension vector, see Theorem \ref{main theorem 3}.

It is worth pointing out that in \cite{CL} we construct a desingularization of arbitrary quiver Grassmannians for finite-dimensional Gorenstein projective modules of $1$-Gorenstein gentle algebras in terms of quiver Grassmannians for their Cohen-Macaulay Auslander algebras.

\vspace{0.2cm} \noindent{\bf Acknowledgments.}
This work is inspired by some discussions with Professor Changjian Fu. The authors thank him very much. The authors deeply thank the referee for very helpful and insightful comments.

The first author(X. Chen) was supported by the National Natural Science Foundation of China (Grant No. 11526168 and No. 11601441) and
the Fundamental Research Funds for the Central Universities (Grant No. 2682016CX109).
The corresponding author(M. Lu) was supported by the National Natural Science Foundation of China (Grant No. 11401401).

\section{Preliminaries}
Throughout this paper, we always assume that $K$ is an algebraically closed field. For any finite set $S$, we denote by $|S|$ the number of the elements in $S$. For a $K$-algebra, we always means a basic finite-dimensional associative $K$-algebra. For any algebra $A$, we denote by $\gldim A$ its \emph{global dimension}. For an additive category $\ca$, we denote by $\ind \ca$ the isomorphism classes of indecomposable objects in $\ca$.

Let $Q=(Q_0,Q_1)$ be a quiver (where $Q_0$ is the set of vertices and $Q_1$ is the set of arrows) and $\langle I\rangle$ an \emph{admissible ideal} in the path algebra $KQ$ which is generated by a set of relations $I$. Denote by $(Q,I)$ the \emph{associated bound quiver}. For any arrow $\alpha$ in $Q$ we denote by $s(\alpha)$ its starting point and by $t(\alpha)$ its ending point. An \emph{oriented path} (or path for short) of \emph{length} $r\geq1$ from $a$ to $b$ is a sequence $p=\alpha_1\alpha_2\dots \alpha_r$ of arrows $\alpha_i$ such that $t(\alpha_i)=s(\alpha_{i-1})$ for all $i=2,\dots,r$, and $s(\alpha_r)=a$, $t(\alpha_1)=b$. A path of length $r\geq1$ is called an \emph{oriented cycle} whenever its source and target coincide. An oriented cycle of length $1$ is called a \emph{loop}.

\subsection{Gentle algebras}
We first recall the definition of special biserial algebras and of gentle algebras.
\begin{definition}[\cite{SW}]
The pair $(Q,I)$ is called special biserial if it satisfies the following conditions.
\begin{itemize}
\item Each vertex of $Q$ is the starting point of at most two arrows, and ending point of at most two arrows.
\item For each arrow $\alpha$ in $Q$ there is at most one arrow $\beta$ such that $\alpha\beta\notin I$, and at most one arrow $\gamma$ such that $\gamma\alpha\notin I$.
\end{itemize}
\end{definition}
\begin{definition}[\cite{AS}]
The pair $(Q,I)$ is called gentle if it is special biserial and moreover the following holds.
\begin{itemize}
\item The set $I$ is generated by zero-relations of length $2$.
\item For each arrow $\alpha$ in $Q$ there is at most one arrow $\beta$ with $t(\beta)=s(\alpha)$ such that $\alpha\beta\in I$, and at most one arrow $\gamma$ with $s(\gamma)=t(\alpha)$ such that $\gamma\alpha\in I$.
\end{itemize}

\end{definition}
A finite-dimensional algebra $A$ is called \emph{special biserial} (resp., \emph{gentle}) if it has a presentation as $A=KQ/\langle I\rangle$ where $(Q,I)$ is special biserial (resp., gentle).

\begin{example}\label{example}
(a) Let $Q$ be the quiver as Figure 1 shows, and $I=\{\beta\alpha,\alpha\gamma_1,\gamma_1\beta\}$. Then $KQ/\langle I\rangle$ is a gentle algebra.

\setlength{\unitlength}{1.2mm}
\begin{center}
\begin{picture}(50,25)

\put(10,10){\circle{1.3}}
\put(30,10){\circle{1.3}}
\put(20,20){\circle{1.3}}
\put(11,11){\vector(1,1){8}}
\put(21,19){\vector(1,-1){8}}
\put(29,10.5){\vector(-1,0){18}}
\put(29,9.5){\vector(-1,0){18}}
\put(9,7){$1$}
\put(19,21){$2$}
\put(29,7){$3$}
\put(13,16){$\alpha$}
\put(25,16){$\beta$}
\put(19,7){$\gamma_2$}
\put(19,12){$\gamma_1$}

\put(-10,0){Figure 1. The quiver $Q$ in Example \ref{example} (a).}
\end{picture}
\vspace{0.2cm}
\end{center}

(b) Let $Q$ be the quiver as Figure 2 shows, and $I=\{\alpha\beta,\beta\alpha,\gamma^2 \}$.
Then $\Lambda=KQ/\langle I\rangle$ is a gentle algebra.

\begin{center}\setlength{\unitlength}{0.7mm}
 \begin{picture}(120,20)
 \put(0,10){\begin{picture}(80,10)
\put(40,0){\circle{2}}
\put(42,1){\vector(1,0){26}}
\put(68,-1){\vector(-1,0){26}}
\put(70,0){\circle{2}}

\qbezier(38,-2)(34,-7)(30,-2)
\qbezier(30,-2)(29,0)(30,2)
\qbezier(38,2)(34,7)(30,2)
\put(37.5,2.5){\vector(1,-1){0.5}}

\put(53,2){$\alpha$}
\put(53,-6){$\beta$}
\put(25,-1){$\gamma$}
\put(39,-6){$1$}
\put(69,-6){$2$}
\end{picture}}

\put(5,-7){Figure 2. The quiver $Q$ in Example \ref{example} (b).}
\end{picture}
\vspace{0.5cm}
\end{center}

\end{example}

A classification of indecomposable modules over gentle algebras can be deduced from the work of Ringel \cite{R} (see e.g. \cite{BR,WW}). For each arrow $\beta$, we denote by $\beta^{-1}$ the formal inverse of $\beta$ with $s(\beta^{-1})=t(\beta)$ and $t(\beta^{-1})=s(\beta)$. A word $w=c_1c_2\cdots c_n$ of arrows and their formal inverse is called a \emph{string} of length $n\geq1$ if $c_{i+1}\neq c_i^{-1}$, $s(c_i)=t(c_{i+1})$ for all $1\leq i\leq n-1$, and no subword nor its inverse is in $I$. We define $(c_1c_2\cdots c_n)^{-1}=c_n^{-1}\cdots c_2^{-1} c_1^{-1}$, and $s(c_1c_2\cdots c_n)=s(c_n)$, $t(c_1c_2\cdots c_n)=t(c_1)$.
We denote the length of $w$ by $l(w)$.
In addition, we also want to have strings of length $0$; be definition, for any vertex $u\in Q_0$, there will be two strings of length $0$, denoted by $1_{(u,1)}$ and $1_{(u,-1)}$, with both
$s(1_{(u,i)})=u=t(1_{(u,i)})$ for $i=-1,1$, and we define $(1_{(u,i)})^{-1}=1_{(u,-i)}$.
We also denote by $\cs(\Lambda)$ the set of all strings over $\Lambda=KQ/\langle I\rangle$.

\begin{remark}\label{lemma string not equal to its inverse}
For any string $w\in\cs(\Lambda)$, we have $w\neq w^{-1}$.
\end{remark}
\begin{proof}
If $w$ is of length zero, then $w=1_{(u,i)}$ for $i=1$ or $-1$, and $w^{-1}=1_{u,-i}$ which is different to $w$ by the definition.

If $l(w)=n\geq1$, then we assume that $w= c_1c_2\cdots c_n$. So $w^{-1}=c_n^{-1}\cdots c_{2}^{-1}c_1^{-1}$. Suppose for a contradiction that $w=w^{-1}$, which means $c_j=c_{n-j+1}^{-1}$ for $j=1,\dots,n$. If $n=2k$ for some integer $k$, then $c_k=c_{k+1}^{-1}$, a contradiction to the definition of strings. If $n=2k+1$ for some integer $k$, then
$c_{k+1}=c_{k+1}^{-1}$, which yields a contradiction. So $w\neq w^{-1}$.
\end{proof}

A \emph{band} $b=\alpha_1\alpha_2\cdots \alpha_{n-1} \alpha_n$ is defined to be a string $b$ with $t(\alpha_1)=s(\alpha_n)$ such that each power $b^m$ is a string, but $b$ itself is not a proper power of any strings. We denote by $\cb(\Lambda)$ the set of all bands over $\Lambda$.

On $\cs(\Lambda)$, we consider the equivalence relation $\rho$ which identifies every string $C$ with its inverse $C^{-1}$. On $\cb(\Lambda)$, we consider the equivalence relation $\rho'$ which identifies every string $C=c_1\dots c_n$ with the cyclically permuted strings $C_{(i)}=c_ic_{i+1}\cdots c_nc_1\cdots c_{i-1}$ and their inverses $C_{(i)}^{-1}$, $1\leq i\leq n$. We choose a complete set $\underline{\cs}(\Lambda)$ of representatives of $\cs(\Lambda)$ relative to $\rho$, and a complete set $\underline{\cb}(\Lambda)$ of representatives of $\cb(\Lambda)$ relative to $\rho'$.

Butler and Ringel showed that each string $w$ defines a unique string module $M(w)$, each band $b$ yields a family of band modules $M(b,m,\phi)$ with $m\geq1$ and $\phi\in \Aut(K^m)$.
Equivalently, one can consider certain quiver morphism $\sigma:S\rightarrow Q$ (for strings) and $\beta:B\rightarrow Q$ (for bands), where $S$ and $B$ are of types $\A_n$ and $\tilde{\A}_n$, respectively. Then string and band modules are given as pushforwards $\sigma_*(M)$ and $\beta_*(R)$ of indecomposable $KS$-modules $M$ and indecomposable regular $KB$-modules $R$, respectively (see e.g. \cite{WW}).
Let $\underline{\Aut}(K^m)$ be a complete set of representatives of indecomposable automorphisms of $K$-spaces with respect to similarity.

\begin{theorem}[\cite{BR}]
The modules $M(w)$ with $w\in \underline{\cs}(\Lambda)$, and the modules $M(b,m,\phi)$ with $b\in \underline{\cb}(\Lambda)$, $m\geq1$ and $\phi\in \underline{\Aut}(K^m)$, provide a complete list of indecomposable (and pairwise non-isomorphic) $\Lambda$-modules.
\end{theorem}

In practice, a string $w$ is of form $\alpha_1^{\epsilon_1} \alpha_2^{\epsilon_2}\cdots \alpha_n^{\epsilon_n}$ for $\alpha_i\in Q_1$ and $\epsilon_i=\pm1$ for all $1\leq i\leq n$. So $w$ can be viewed as a walk in $Q$:
\[\xymatrix{ w:\quad b_1 \ar@{-}[r]^{\quad\alpha_1} &b_2 \ar@{-}[r]^{\alpha_2}& \cdots \ar@{-}[r]^{\alpha_{n-1}}& b_n\ar@{-}[r]^{\alpha_n} &b_{n+1},}\]
where $b_1,b_2\dots,b_{n+1}$ are vertices of $Q$ and $\alpha_i$ is an arrow from $b_{i+1}$ to $b_i$ if $\epsilon_i=1$, or an arrow from $b_i$ to $b_{i+1}$ if $\epsilon_i=-1$, for each $1\leq i\leq n$ . In this way, the equivalence relation $\rho$ induces that
\[\xymatrix{ w:\quad b_1 \ar@{-}[r]^{\quad\alpha_1} &b_2 \ar@{-}[r]^{\alpha_2}& \cdots \ar@{-}[r]^{\alpha_{n-1}}& b_n\ar@{-}[r]^{\alpha_n\quad} &b_{n+1},}\]
is equivalent to
\[\xymatrix{ w^{-1}:\quad b_{n+1} \ar@{-}[r]^{\quad\quad\quad\alpha_n} &b_n \ar@{-}[r]^{\alpha_{n-1}}& \cdots \ar@{-}[r]^{\alpha_{2}}& b_2\ar@{-}[r]^{\alpha_1} &b_1.}\]
It is similar to interpret $\rho'$ if $w$ is a band. We denote by $v\sim w$ for any two strings $v,w$ if $v$ is equivalent to $w$ under $\rho$.

For any string $w=c_1\dots c_n$, or $w=1_{(u,j)}$, let $u_w(i)=t(c_{i+1})$, $0\leq i<n$, and $u_w(n)=s(w)=s(c_n)$. Given a vertex $v\in Q_0$, let $I_w(v)=\{ i|u_w(i)=v\}\subseteq\{0,1,\dots,n\}$.
Denote by $k_w(v)=|I_w(v)|$.
We associate a vector $(k_w(v) )_{v\in Q_0}$ to the string $w$, which is denoted by $\dimv w$, and call it the \emph{dimension vector} of $w$.
From \cite{BR}, we get that $\dimv w=\dimv M(w)$.

Note that if a gentle algebra $\Lambda$ is representation-finite, then there is no band module in $\mod \Lambda$, and so all the indecomposable modules over $\Lambda$ are string modules.

\subsection{Singularity categories and Gorenstein algebras}
Let $\Lambda$ be a finite-dimensional $K$-algebra. Let $\mod \Lambda$ be the category of finitely generated left $\Lambda$-modules, and $\proj \Lambda$ the subcategory of finitely generated projective $\Lambda$-modules. For an arbitrary $\Lambda$-module $_\Lambda X$, we denote by $\pd_\Lambda X$ (resp. $\id_\Lambda X$) the projective dimension (resp. the injective dimension) of the module $_\Lambda X$. A $\Lambda$-module $G$ is \emph{Gorenstein projective}, if there is an exact sequence $$P^\bullet:\cdots \rightarrow P^{-1}\rightarrow P^0\xrightarrow{d^0}P^1\rightarrow \cdots$$ of projective $\Lambda$-modules, which stays exact under $\Hom_\Lambda(-,\Lambda)$, and such that $G\cong \Ker d^0$. We denote by $\Gproj(\Lambda)$ the subcategory of Gorenstein projective $\Lambda$-modules.

\begin{definition}[\cite{AR1,AR2,Ha1}]
A finite-dimensional algebra $\Lambda$ is called a Gorenstein (or Iwanaga-Gorenstein) algebra if $\Lambda$ satisfies $\id \Lambda_\Lambda<\infty$ and $\id_\Lambda \Lambda<\infty$.
\end{definition}

Observe that for a Gorenstein algebra $\Lambda$, we have $\id _\Lambda\Lambda=\id \Lambda_\Lambda$, see e.g. \cite[Lemma 6.9]{Ha1}; the common value is denoted by $\Gd \Lambda$. If $\Gd \Lambda\leq d$, we say that $\Lambda$ is \emph{$d$-Gorenstein}.

For an algebra $\Lambda$, the \emph{singularity category} of $\Lambda$ is defined to be the quotient category $D_{sg}^b(\Lambda):=D^b(\Lambda)/K^b(\proj \Lambda)$ \cite{Bu,Ha1,Or1}. Note that $D_{sg}^b(\Lambda)$ is zero if and only if $\gldim \Lambda<\infty$ \cite{Ha1}.

\begin{theorem}[\cite{Bu,Ha1}]
Let $\Lambda$ be a finite-dimensional algebra. Then $\Gproj (\Lambda)$ is a Frobenius category with the projective modules as the projective-injective objects. If $\Lambda$ is Gorenstein, then the stable category $\underline{\Gproj}(\Lambda)$ is triangle equivalent to the singularity category $D^b_{sg}(\Lambda)$ of $\Lambda$.
\end{theorem}

An algebra is of \emph{finite Cohen-Macaulay type}, or simply, \emph{CM-finite}, if there are only finitely many isomorphism classes of indecomposable finitely generated Gorenstein projecitve modules. Clearly, $\Lambda$ is CM-finite if and only if there is a finitely generated module $E$ such that $\Gproj \Lambda=\add E$. In this way, $E$ is called to be a \emph{Gorenstein projective generator}. If the global dimension of $\Lambda$ is finite, then $\Gproj \Lambda=\proj \Lambda$, which implies that $\Lambda$ is CM-finite. If $\Lambda$ is self-injective, then $\Gproj \Lambda=\mod \Lambda$, so $\Lambda$ is CM-finite if and only if $\Lambda$ is representation-finite.

Let $\Lambda$ be a CM-finite algebra, $E_1,\dots,E_n$ all the pairwise non-isomorphic indecomposable Gorenstein projective $\Lambda$-modules. Put $E=\oplus_{i=1}^n E_i$. Then $E$ is a Gorenstein projective generator. We call $\Aus(\Gproj \Lambda):=(\End_\Lambda E)^{op}$ the \emph{Cohen-Macaulay Auslander algebra} (also called \emph{relative Auslander algebra}) of $\Lambda$.

Gei{\ss} and Reiten \cite{GR} prove that gentle algebras are Gorenstein algebras, so their Cohen-Macaulay Auslander algebras have finite global dimensions \cite{LZ}. The singularity category of a gentle algebra is characterized by Kalck in \cite{Ka}, we recall it as follows.
For a gentle algebra $\Lambda=KQ/\langle I\rangle$, we denote by $\cc(\Lambda)$ the set of equivalence classes (with respect to cyclic permutation) of \emph{repetition-free} cyclic paths $\alpha_1\dots\alpha_n$ in $Q$ such that $\alpha_i\alpha_{i+1}\in I$ for all $i$, where we set $n+1=1$. Moreover, we set $l(c)$ to be the \emph{length} of the cycle $c\in\cc(\Lambda)$, i.e. $l(\alpha_1\dots\alpha_n)=n$.

For every arrow $\alpha\in Q_1$, there is at most one cycle $c\in\cc(\Lambda)$ containing $\alpha$. In fact, if there are two different elements $c,c'\in \cc(\Lambda)$ such that $\alpha$ lies on both of them, then the definition of $\cc(\Lambda)$ implies that there exist arrows $\beta$, $\gamma_1$ and $\gamma_2$ such that $\gamma_1\neq \gamma_2$, $s(\gamma_1)=t(\beta)=s(\gamma_2)$ and $\gamma_1\beta,\gamma_2\beta\in I$, a contradiction to that $\Lambda$ is gentle.  We define $R(\alpha)$ to be the \emph{left ideal} $\Lambda \alpha$ generated by $\alpha$.
It follows from the definition of gentle algebras that this is a direct summand of the radical $\rad P_{s(\alpha)}$ of the indecomposable projective $\Lambda$-module $P_{s(\alpha)}=\Lambda e_{s(\alpha)}$, where $e_{s(\alpha)}$ is the idempotent corresponding to $s(\alpha)$. In fact, all radical summands of indecomposable projective modules arise in this way, see e.g. \cite{Ka}.

\begin{theorem}[\cite{Ka}]\label{theorem Kalck}
Let $\Lambda=KQ/\langle I\rangle$ be a gentle algebra. Then

\emph{(i)} $\ind \Gproj(\Lambda)=\ind \proj \Lambda \bigcup \{R(\alpha_1),\dots, R(\alpha_n)|c=\alpha_1\cdots \alpha_n\in\cc(\Lambda)\}$.

\emph{(ii)} There is an equivalence of triangulated categories
$$D^b_{sg}(\Lambda)\simeq \prod_{c\in\cc(\Lambda)} \frac{D^b( K)}{[l(c)]} ,$$
where $ D^b(K)/[l(c)]$ denotes the triangulated orbit category, see \cite{Ke}.
\end{theorem}

From Theorem \ref{theorem Kalck} or its proof in \cite{Ka}, we get that $\underline{\Gproj}(\Lambda)\simeq D^b_{sg}(\Lambda)$ is equivalent to a semisimple abelian category and therefore itself is semisimple abelian. In particular, $\underline{\Hom}_\Lambda(R(\alpha),R(\alpha'))\cong \delta_{\alpha \alpha'}K$ for any two non-projective indecomposable Gorenstein projective modules $R(\alpha),R(\alpha')$.

\section{Cohen-Macaulay Auslander algebras of gentle algebras}
Let $\Lambda=KQ/\langle I\rangle$ be a gentle algebra. It is easy to get the following lemma.
\begin{lemma}
Let $\Lambda=KQ/\langle I\rangle$ be a gentle algebra. Then $\Lambda$ is CM-finite.
\end{lemma}
\begin{proof}
From Theorem \ref{theorem Kalck}, we get that
$$\ind \Gproj(\Lambda)=\ind \proj \Lambda \bigcup \{R(\alpha_1),\dots, R(\alpha_n)|c=\alpha_1\cdots \alpha_n\in\cc(\Lambda)\}.$$
So every non-projective indecomposable Gorenstein projective $\Lambda$-module is of form $R(\alpha)$ for some arrow $\alpha$. Furthermore, there are only finitely many arrows, and then $\Lambda$ is CM-finite.
\end{proof}

From $\Lambda$, we construct a bound quiver $(Q^{Aus},I^{Aus})$ as follows:

$\bullet$ the set of vertices $Q^{Aus}_0:=Q_0\bigsqcup Q_1^{cyc}$, where $Q_1^{cyc}=\{\alpha| \exists\,c\in\cc(\Lambda) \mbox{ such that } \alpha \mbox{ lies on } c\}$;

$\bullet$ the set of arrows $Q^{Aus}_1:=Q_1^{ncyc}\bigsqcup (Q_1^{cyc})^{\pm}$, where $Q_1^{ncyc}=Q_1\setminus Q_1^{cyc}$ (i.e. arrows do not lie on any cyclic paths in $\cc(\Lambda)$), $(Q_1^{cyc})^{+}=\{\alpha^+:s(\alpha)\rightarrow \alpha | \alpha\in Q_1^{cyc} \}$ and
$(Q_1^{cyc})^{-}=\{\alpha^-:\alpha\rightarrow t(\alpha) | \alpha\in Q_1^{cyc} \}$.

$\bullet$ the set of relations $I^{Aus}:= \{\beta^+\alpha^-| \beta\alpha \in I\mbox{ with }\alpha,\beta\in Q_1^{cyc}\}\bigcup\{\beta\alpha|\beta\alpha\in I\mbox{ with }\alpha,\beta\in Q_1^{ncyc}\}$.

Note that if $\cc(\Lambda)=\emptyset$, then $(Q^{Aus},I^{Aus})=(Q,I)$.

In this section, we prove that $KQ^{Aus}/ \langle I^{Aus}\rangle$ is isomorphic to the Cohen-Macaulay Auslander algebra of the gentle algebra $\Lambda=KQ/\langle I\rangle$.

\begin{example}\label{example 3}
(a) Keep the notations as in Example \ref{example} (a). Then the quiver $Q^{Aus}$ of the gentle algebra $KQ/\langle I\rangle$ is as Figure 3 shows, and $I^{Aus}=\{\alpha^+\gamma_1^-,\beta^+\alpha^-,\gamma_1^+\beta^-\}$.
\setlength{\unitlength}{1.2mm}
\begin{center}
\begin{picture}(220,25)

\put(50,10){\circle{1.3}}
\put(50,11){\vector(0,1){8}}
\put(60,25){\circle{1.3}}
\put(60,5){\circle{1.3}}
\put(51,20.5){\vector(2,1){8}}
\put(50,20){\circle{1.3}}
\put(61,24.5){\vector(2,-1){8}}

\put(70,10){\circle{1.3}}
\put(69,10){\vector(-1,0){18}}
\put(69,9.5){\vector(-2,-1){8}}
\put(70,20){\circle{1.3}}
\put(70,19){\vector(0,-1){8}}
\put(59,5.5){\vector(-2,1){8}}

\put(47,9){\tiny$1$}
\put(47,19){\tiny$\alpha$}
\put(59,26){\tiny$2$}
\put(59,2.5){\tiny$\gamma_1$}
\put(71,19){\tiny$\beta$}
\put(71,9){\tiny$3$}
\put(50.5,14){\tiny$\alpha^+$}
\put(54,20.5){\tiny$\alpha^-$}
\put(62.5,20){\tiny$\beta^+$}
\put(66.5,14){\tiny$\beta^-$}
\put(62.5,8){\tiny$\gamma_1^+$}
\put(54.5,8){\tiny$\gamma_1^-$}
\put(58,11){\tiny $\gamma_2$}
\put(20,-3){Figure 3. The quiver $Q^{Aus}$ of $KQ/\langle I\rangle$ for Example \ref{example} (a).}
\end{picture}
\vspace{0.2cm}
\end{center}

(b) Keep the notations as in Example \ref{example} (b). Then the quiver $Q^{Aus}$ of the gentle algebra $KQ/\langle I\rangle$ is as Figure 4 shows, and $I^{Aus}=\{ \gamma^+\gamma^-,\alpha^+\beta^-, \beta^+\alpha^- \}$.

\begin{center}\setlength{\unitlength}{0.7mm}
 \begin{picture}(150,32)
 \put(0,10){\begin{picture}(80,10)
\put(40,0){\circle{2}}
\put(42,1){\vector(1,0){26}}
\put(68,-1){\vector(-1,0){26}}
\put(70,0){\circle{2}}

\put(72,2){\vector(1,1){12}}
\put(86,16){\circle{2}}
\put(86,-16){\circle{2}}

\put(88,14){\vector(1,-1){12}}
\put(102,0){\circle{2}}
\put(100,-2){\vector(-1,-1){12}}
\put(84,-14){\vector(-1,1){12}}

\put(53,2){$\gamma^-$}
\put(53,-6){$\gamma^+$}
\put(35,-1){$\gamma$}
\put(73,8){$\alpha^+$}
\put(84,18){$\alpha$}
\put(72,-12){$\beta^-$}
\put(93,-12){$\beta^+$}
\put(94,8){$\alpha^-$}
\put(69,-6){$1$}
\put(104,-1){$2$}
\put(84,-22){$\beta$}
\end{picture}}

\put(15,-21){Figure 4. The quiver $Q^{Aus}$ of $KQ/\langle I\rangle$ for Example \ref{example} (b).}
\end{picture}
\vspace{1.5cm}
\end{center}
\end{example}

For any two $\Lambda$-modules $M,N$ and any subcategory $\cd$ of $\mod \Lambda$ containing $M,N$, we denote by $\irr_{\cd}(M,N)$ the space of irreducible morphisms from $M$ to $N$ in $\cd$.

From Theorem \ref{theorem Kalck}, we get that
$$\ind \Gproj(\Lambda)=\ind \proj \Lambda \bigcup \{R(\alpha_1),\dots, R(\alpha_n)|c=\alpha_1\cdots \alpha_n\in\cc(\Lambda)\}.$$
Furthermore, let $c\in\cc(\Lambda)$ be a cycle, which we label as follows: $1\xrightarrow{\alpha_1} 2\xrightarrow{\alpha_2} \cdots \xrightarrow{\alpha_{n-1}}n\xrightarrow{\alpha_n}1$.
Then from the proof of \cite[Theorem 2.5]{Ka}, there are short exact sequences
\begin{equation}\label{equation 1}
0\rightarrow R(\alpha_{i})\xrightarrow{a_i} P_i\xrightarrow{b_i} R(\alpha_{i-1}) \rightarrow0,
\end{equation}
for all $i=1,\dots,n$, where we set $\alpha_0=\alpha_n$.

\begin{lemma}\label{lemma irreducible morphism of Gorenstein projective modules}
Keep the notations as above. Then $a_i,b_i$ in sequence (\ref{equation 1}) are irreducible morphisms in $\Gproj\Lambda$ for all $i=1,\dots,n$. Furthermore,

\emph{(i)} $$\dim_K\irr_{\Gproj\Lambda}(R(\alpha_i),P_i))=1 \mbox{ and } \dim_K\irr_{\Gproj\Lambda}(P_i,R(\alpha_{i-1}))=1,$$
for all $i=1,\dots,n$.

\emph{(ii)} For any indecomposable projective module $P$ not isomorphic to $P_i$, we have $$\irr_{\Gproj\Lambda}(R(\alpha_i),P))=0 \mbox{ and } \irr_{\Gproj\Lambda}(P,R(\alpha_{i-1}))=0,$$
for all $i=1,\dots,n$.

\emph{(iii)} For any two non-projective indecomposable Gorenstein projective modules $R(\alpha)$ and $R(\alpha')$, we have $\irr_{\Gproj\Lambda}(R(\alpha),R(\alpha'))=0$.
\end{lemma}
\begin{proof}
Note that $R(\alpha_i)$ is indecomposable and sequence (\ref{equation 1}) is not split for any $\alpha_n\dots\alpha_1\in\cc(\Lambda)$ and each $i=1,\dots,n$. We need to check that sequence (\ref{equation 1}) is an almost split sequence in $\Gproj\Lambda$ for each $i=1,\dots,n$.

For any Gorenstein projective module $M$, and a morphism $v:M\rightarrow R(\alpha_{i-1})$ which is not a retraction, since $\underline{\Gproj}(\Lambda)$ is a semisimple category, and $R(\alpha_{i-1})$ is a simple object in $\underline{\Gproj}(\Lambda)$, we get that $v=0$ in $\underline{\Gproj}(\Lambda)$.
So $v$ factors through a projective module
$P$ as $v=v_2v_1$ for some morphisms $v_1:M\rightarrow P$ and $v_2:P\rightarrow R(\alpha_{i-1})$.
It is easy to see that $v_2$ factors through $b_i$ as $v_2=b_iv_3$ for some morphism $v_3:P\rightarrow P_i$, which implies $v=v_2v_1=b_iv_3v_1$, so $b_i$ is right almost split and then sequence (\ref{equation 1})
is almost split.
\[\xymatrix{ R(\alpha_i) \ar[r]^{a_i} & P_i \ar[r]^{b_i} & R(\alpha_{i-1}) \\
& P\ar@{.>}[u]^{v_3} \ar@{.>}[ur]^{v_2} & M.\ar[u]^v\ar@{.>}[l]^{v_1}}\]

(i) For any other irreducible morphism $a_i':R(\alpha_i)\rightarrow P_i$, since $\Ext^1_\Lambda(R(\alpha_{i-1}),P_i)=0$, there exists
a morphism $f:P_i\rightarrow P_i$ such that $a_i'=fa_i$. Note that $a_i$ is not a section, so $f$ is a retraction and then an isomorphism, so $\dim_K\irr_{\Gproj\Lambda}(R(\alpha_i),P_i))=1$.

It is similar to prove that $\dim_K\irr_{\Gproj\Lambda}(P_i,R(\alpha_{i-1}))=1$, we omit the proof here.

(ii) follows from that sequence (\ref{equation 1}) is almost split.

(iii) If $\alpha\neq \alpha'$, then $\underline{\Hom}_\Lambda(R(\alpha),R(\alpha'))=0$, so $\irr_{\Gproj\Lambda}(R(\alpha),R(\alpha'))=0$.
If $\alpha=\alpha'$, then by the proof of Theorem \ref{theorem Kalck} in \cite{Ka}, we get that $\End_\Lambda(R(\alpha))=K$. So $\irr_{\Gproj\Lambda}(R(\alpha),R(\alpha))=0$.
\end{proof}

Since $\proj\Lambda\subset \Gproj\Lambda$, for any indecomposable projective $\Lambda$-modules $P_1,P_2$, we get that $\irr_{\Gproj\Lambda}(P_1,P_2)\subseteq \irr_{\proj\Lambda}(P_1,P_2)$.

\begin{lemma}\label{lemma irreducible morphism of projective modules 2}
Let $\Lambda=KQ/\langle I\rangle$ be a gentle algebra. Let $P_1,P_2$ be two indecomposable projective $\Lambda$-modules with their corresponding vertices $v_1,v_2$ respectively. For any irreducible morphism $f:P_1\rightarrow P_2$ in $\proj\Lambda$ which is induced by an arrow $\alpha:v_2\rightarrow v_1$, then

\emph{(i)} if $\alpha$ lies on a cycle in $\cc(\Lambda)$, then $f$ is not irreducible in $\Gproj\Lambda$, in particular, $f$ factors through $R(\alpha)$ as a composition of two irreducible morphisms in $\Gproj\Lambda$.

\emph{(ii)} if $\alpha$ does not lie on any cycle in $\cc(\Lambda)$, then $f$ is irreducible in $\Gproj\Lambda$.
\end{lemma}
\begin{proof}
(i) If $\alpha$ lies on a cycle $c\in\cc(\Lambda)$, we assume that $c$ is of form $\cdots v_3 \xrightarrow{\gamma}v_2\xrightarrow{\alpha} v_1\xrightarrow{\beta}0 \cdots$ (where the vertices can be coincided), then there exist two short exact sequences
$$0\rightarrow R(\alpha)\xrightarrow{a_1} P_2 \xrightarrow{b_1} R(\gamma)\rightarrow0 \mbox{ and } 0\rightarrow R(\beta)\xrightarrow{a_2} P_1 \xrightarrow{b_2}R(\alpha)\rightarrow0,$$
with $a_1b_2=f$. So $f$ is not irreducible in $\Gproj\Lambda$. Lemma \ref{lemma irreducible morphism of Gorenstein projective modules} yields that $a_1,b_2$ are irreducible in $\Gproj\Lambda$, and then (i) follows.

(ii)
Since $f\in \irr_{\proj\Lambda}(P_1,P_2)$, we get that $f$ is neither a section nor a retraction. Suppose for a contradiction that $f$ factors through a module $M\in\Gproj\Lambda$ as $f=f_2f_1$ for some morphisms $f_1:P_1\rightarrow M$ and $f_2:M\rightarrow P_2$, with neither $f_1$ a section nor $f_2$ a retraction. Then $M\notin \proj\Lambda$, so $M=M_1\oplus M_2$ with $M_1$ projective and the indecomposable direct summands of $M_2$ non-projective. Note that $M_2\neq0$.
For any non-projective indecomposable Gorenstein projective module $R_i$, there exist indecomposable projective modules $P_i,P_{i+1}$ and non-projective Gorenstein projective modules $R_{i-1},R_{i+1}$ such that
the following sequences are exact
\begin{equation}\label{equation 2}
0\rightarrow R_i\rightarrow P_i\rightarrow R_{i-1}\rightarrow0, \quad 0\rightarrow R_{i+1}\rightarrow P_{i+1}\rightarrow R_i\rightarrow0.
\end{equation}
So by doing direct sum of the exact sequences as in sequence (\ref{equation 2}) for all indecomposable direct summands of $M_2$, there exist two exact sequences
\begin{equation}\label{equation 3}
0\rightarrow N_1\xrightarrow{a_1} P_{M_2} \xrightarrow{b_1} M_2\rightarrow0,\quad 0\rightarrow M_2\xrightarrow{a_2} Q_{M_2}\xrightarrow{b_2} N_2\rightarrow0,
\end{equation}
where $P_{M_2},Q_{M_2}$ are projective with their indecomposable direct summands corresponding to vertices lying on cycles in $\cc(\Lambda)$, and $N_1,N_2$ are Gorenstein projective modules with their indecomposable direct summands non-projective.
Then for $M$, there exist two exact sequences
\begin{equation}\label{equation 4}
0\rightarrow N_1\xrightarrow{c_1} M_1\oplus P_{M_2} \xrightarrow{d_1} M\rightarrow0,\quad 0\rightarrow M\xrightarrow{c_2} M_1\oplus Q_{M_2}\xrightarrow{d_2} N_2\rightarrow0.
\end{equation}
The proof can be broken into the following two cases.

{\bf Case (1). The vertex $v_1$ does not lie on any cycle in $\cc(\Lambda)$.} Then $f_1$ factors through $d_1$ as the following diagram shows:
\[\xymatrix{ & P_1\ar@{.>}[d]^{f_1'} \ar[rd]^{f_1}\ar[r]^f &P_2\\
N_1\ar[r]^{c_1\quad\quad} &M_1\oplus P_{M_2} \ar[r]^{d_1} &M.\ar[u]_{f_2}}\]
So $f=f_2d_1f_1'$. If $f_1'$ is not a section, then $f_2d_1$ is a retraction since $f$ is irreducible in $\proj \Lambda$ and $M_1\oplus P_{M_2}$ is projective, which yields that $f_2$ is a retraction, giving a contradiction.
So $f_1'$ is a section, which implies that $P_1$ is a direct summand of $M_1$ by the assumption that the vertex $1$ does not lie on any cycle in $\cc(\Lambda)$. Since $M_1$ is a direct summand of $M$, we get that $P_1$ is a direct summand of $M$, i.e. $f_1$ is a section, giving a contradiction.

{\bf Case (2). The vertex $v_1$ lies on some cycle in $\cc(\Lambda)$.}
Then there is a cycle $c\in\cc(\Lambda)$ such that $v_1$ lies on $c$.
So we assume that $c$ locally is $\cdots \xrightarrow{\alpha_1} v_3\xrightarrow{\alpha_2} v_1\xrightarrow{\alpha_3}\cdots$. Let $P_3$ be the indecomposable projective module corresponding to the vertex $v_3$.
Then
there are two exact sequences:
\begin{equation}\label{equation 5}
0\rightarrow R(\alpha_2)\xrightarrow{u_1} P_3 \xrightarrow{v_1} R(\alpha_1)\rightarrow0,\quad 0\rightarrow R(\alpha_3)\xrightarrow{u_2} P_1 \xrightarrow{v_2} R(\alpha_2)\rightarrow0.
\end{equation}
Similar to Case (1), we get that $f_1$ factors through $d_1$ as the following diagram shows:
\[\xymatrix{ & P_1\ar@{.>}[d]^{f_1'} \ar[rd]^{f_1}\ar[r]^f &P_2\\
N_1\ar[r]^{c_1\quad\quad} &M_1\oplus P_{M_2} \ar[r]^{d_1} &M.\ar[u]_{f_2}}\]
Then $f=f_2d_1f_1'$. If $f_1'$ is not a section, then $f_2d_1$ is a retraction since $f$ is irreducible in $\proj \Lambda$ and $M_1\oplus P_{M_2}$ is projective, which yields that $f_2$ is a retraction, giving a contradiction.
So $f_1'$ is a section.

If $f_1'$ induces that $P_1$ is a direct summand of $M_1$, and then it is a direct summand of $M$. By our construction, we get that $f_1:P_1\rightarrow M$ is a section, giving a contradiction. So $f_1'$ induces that $P_1$ is a direct summand of $P_{M_2}$. By our construction, we know that $R(\alpha_2)$ is a direct summand of $M_2$. So $f$ factors through $v_2:P_1\rightarrow R(\alpha_2)$ as $f=g_2v_2$ for some morphism $g_2:R(\alpha_2)\rightarrow P_2$.
From sequence (\ref{equation 5}), we get that $g_2$ factors through $u_1$ as $g_2=g_2'u_1$ for some morphism $f_2':P_3\rightarrow P_2$ since $\Ext^1_\Lambda(R(\alpha_1),P_2)=0$.
Then $f=f_2'u_1v_2$. Since $u_1v_2:P_1\rightarrow P_3$ is the morphism induced by the arrow $\alpha_2$, it is not a section. Therefore, $f_2'$ is a retraction and then an isomorphism. So $f$ is the morphism induced by the arrow $\alpha_2$. However, $f$ is the morphism induced by the arrow $\alpha$, so $\alpha_2=\alpha$. Recall that $\alpha$ does not lie on any cycle in $\cc(\Lambda)$, giving a contradiction.

To sum up, $f$ is an irreducible morphism in $\Gproj(\Lambda)$.
\end{proof}

\begin{theorem}\label{main theorem 1}
Let $\Lambda=KQ/\langle I\rangle$ be a gentle algebra. Then the Cohen-Macaulay Auslander algebra of $\Lambda$ is isomorphic to $KQ^{Aus}/\langle I^{Aus}\rangle$.
\end{theorem}
\begin{proof}
Note that
$$\ind \Gproj(\Lambda)=\ind \proj \Lambda \bigcup \{R(\alpha_1),\dots, R(\alpha_n)|c=\alpha_1\cdots \alpha_n\in\cc(\Lambda)\}.$$
Lemma \ref{lemma irreducible morphism of Gorenstein projective modules} and Lemma \ref{lemma irreducible morphism of projective modules 2} characterize all the irreducible morphisms in $\Gproj\Lambda$, from them, it is easy to see that $Q^{Aus}$ is the quiver of the Cohen-Macaulay Auslander algebra of $\Lambda$.
In fact, the vertex $i\in Q_0\subseteq Q_0^{Aus}$ corresponds to the corresponding indecomposable projective $\Lambda$-module $P_i$; the vertex $\alpha\in Q_1^{cyc}\subseteq Q^{Aus}_0$ corresponding to the $\Lambda$-module $R(\alpha)$; the arrow $\beta\in Q_1^{ncyc}\subseteq Q_1^{Aus}$ corresponds to the irreducible morphism $P_{t(\beta)}\rightarrow P_{s(\beta)}$ induced by $\beta\in Q_1$, see Lemma \ref{lemma irreducible morphism of projective modules 2} (ii). The arrow $\alpha^{-}$ (resp. $\alpha^+$) corresponds to the irreducible morphism $P_{t(\alpha)}\xrightarrow{b} R(\alpha)$ (resp. $R(\alpha)\xrightarrow{a} P_{s(\alpha)}$), see Lemma \ref{lemma irreducible morphism of Gorenstein projective modules} and Lemma \ref{lemma irreducible morphism of projective modules 2} (i). Note that $b$ is surjective and $a$ is injective.

So $\Aus(\Gproj\Lambda)$ is isomorphic to $KQ^{Aus}/\langle I^{A}\rangle$ for some admissible ideal $\langle I^{A}\rangle$. Recall that
$$I^{Aus}= \{\beta^+\alpha^-| \beta\alpha \in I\mbox{ with }\alpha,\beta\in Q_1^{cyc}\}\bigcup\{\beta\alpha|\beta\alpha\in I \mbox{ with }\alpha,\beta\in Q_1^{ncyc}\}.$$
From the above, it is easy to see that $\langle I^{Aus}\rangle\subseteq \langle I^A\rangle$.
Assume that $l=\sum_{i=1}^t k_il_i\in I^A$, where $l_1,\dots,l_t$ are paths in $KQ^{Aus}$ and $k_i\neq0$ for $1\leq i\leq t$. We can also assume that
the starting points and the ending points of all the $l_i,1\leq i\leq t$ are same, which are denoted by $s(l)$, $t(l)$ respectively.
The proof can be broken into the following four cases.

{\bf Case (1).} $s(l),t(l)\in Q_0\subseteq Q_0^{Aus}$. We can view $l$ to be an element in $KQ$ after replacing all the subpaths $\alpha^-\alpha^+$ by $\alpha$, and denote it by $\pi(l)$.  Let us view the arrows as irreducible morphisms. For any arrow $\alpha\in Q_1^{cyc}$, the irreducible morphism from $P_{t(\alpha)}$ to $P_{s(\alpha)}$ in $\proj\Lambda$ induced by $\alpha$ is equal to the combination of the irreducible morphisms in $\Gproj \Lambda$ induced by the arrows $\alpha^-$ and $\alpha^+$. So the morphism from $P_{t(l)}$ to $P_{s(l)}$ induced by $\pi(l)$ in $\proj\Lambda$ is equal to the one induced by $l$ in $\Gproj\Lambda$. Since $l\in I^A$, the morphism from $P_{t(l)}$ to $P_{s(l)}$ induced by $l$ is zero, and then the morphism from $P_{t(l)}$ to $P_{s(l)}$ induced by $\pi(l)$ is also zero. So $\pi(l)\in \langle I\rangle$, and then $\pi(l_i)\in \langle I\rangle$ for any $1\leq i\leq t$, since $\langle I\rangle$ is generated by zero-relations of length two. In other words, for each $1\leq i\leq t$, there exist two arrows $\alpha$, $\beta$ in $Q$ such that $\beta\alpha\in I$ and $\beta\alpha$ is a subpath of $\pi(l_i)$. If $\alpha\in Q_1^{ncyc}$, then $\beta\in Q_1^{ncyc}$, and so $\beta\alpha \in I^{Aus}$, which implies that $l_i\in \langle I^{Aus}\rangle$; if $\alpha\in Q_1^{cyc}$, then $\beta\in Q_1^{cyc}$ and so $\beta^+\alpha^-\in I^{Aus}$. It is easy to see that $\beta^+\alpha^-$ is a subpath of $l_i$, which implies that $l_i\in \langle I^{Aus}\rangle$. Therefore, we have $l_i\in \langle I^{Aus}\rangle$ for each $i$, and then $l\in \langle I^{Aus}\rangle$.

{\bf Case (2).} $s(l)=\alpha \in Q_1^{cyc}\subseteq Q_0^{Aus}, t(l)\in Q_0\subseteq Q_0^{Aus}$. Since there is only one arrow $\alpha^-$ starting from $\alpha$, we can assume $l=l'\alpha^-$ where $l'$ is some element in $KQ^{Aus}$ starting from $t(\alpha)$. Viewing the arrows as irreducible morphisms, since $\alpha^+$ corresponds to an injective morphism, we get that $l=l'\alpha^-\in \langle I^A\rangle$ if and only if $l\alpha^+\in \langle I^A\rangle$. Then $l\alpha^+$ satisfies Case (1), which implies that it is in $\langle I^{Aus}\rangle$. Since $\langle I^{Aus}\rangle$ is generated by zero-relations of length two and $\alpha^-\alpha^+\notin \langle I^{Aus}\rangle$, we get that $l\in \langle I^{Aus}\rangle$.

{\bf Case (3).} $s(l)\in Q_0\subseteq Q_0^{Aus}, t(l)=\alpha \in Q_1^{cyc}\subseteq Q_0^{Aus}$. It is similar to Case (2), only need note that $\alpha^-$ corresponds to a surjective morphism.

{\bf Case (4).} $s(l)=\alpha,t(l)=\beta \in Q_1^{cyc}\subseteq Q_0^{Aus}$. It is also similar to Case (2), only need note that $\alpha^+$ corresponds to an injective morphism and $\beta^-$ corresponds to a surjective morphism.

Therefore, $\langle I^{Aus}\rangle =\langle I^A\rangle$, and so $KQ^{Aus}/\langle I^{Aus}\rangle$ is isomorphic to the Cohen-Macaulay Auslander algebra of $\Lambda$.
\end{proof}

\begin{corollary}
Let $\Lambda=KQ/\langle I\rangle$ be a gentle algebra. Then the Cohen-Macaulay Auslander algebra of $\Lambda$ is also a gentle algebra.
\end{corollary}
\begin{proof}
From the structure of $Q^{Aus}$ and $I^{Aus}$, it is easy to see that $KQ^{Aus}/\langle I^{Aus}\rangle$ is a gentle algebra.
\end{proof}

\section{Some representation properties of the Cohen-Macaulay Auslander agelbras for gentle algebras}

Before going on, let us fix some notations.
Let $\Lambda$ be a gentle algebra and $\Gamma$ be its Cohen-Macaulay Auslander algebra.

For any $M=((M_i)_{i\in Q_0},(M_\alpha:M_i\rightarrow M_j)_{(\alpha:i\rightarrow j)\in Q_1})\in\mod \Lambda$,
define a $\Gamma$-module $\widehat{M}=((N_i,N_\alpha)_{i\in Q_0,\alpha\in Q_1^{cyc}},(N_\beta)_{\beta \in Q_1^{Aus}})$ as follows:

$\bullet$ For any $i\in Q_0\subseteq Q_0^{Aus}$, we set $N_i=M_i$; for any $\alpha\in Q_1^{cyc}\subseteq Q_0^{Aus}$, we set $N_\alpha=\Im M_\alpha$.

$\bullet$ For any arrow in $Q_1^{Aus}$, if it is of form $(\beta:i\rightarrow j)\in Q_1^{ncyc}$, then we set $N_\beta=M_\beta$; if it is of form $\beta^+:i\rightarrow \beta$, or of form
$\beta^-:\beta\rightarrow j$ for some $(\beta:i\rightarrow j)\in Q_1^{cyc}$, we set $N_{\beta^+}$ and $N_{\beta^-}$ to be the natural morphisms $(N_i=M_i)\rightarrow (\Im M_\beta=N_{\beta})$ and $(N_\beta=\Im M_\beta)\rightarrow (M_j=N_j)$ respectively, which are induced by $M_\beta:M_i\rightarrow M_j$.

It is easy to see that $\widehat{M}$ is actually a $\Gamma$-module. Since $\Im$ is a functor, we can define a functor $\Phi:\mod \Lambda\rightarrow \mod \Gamma$ such that $\Phi(M):=\widehat{M}$, with the natural definition on morphisms.

\begin{lemma}[\cite{CL}]\label{lemma additive functor of phi}
Keep the notations as above. Then $\Phi$ is a covariant additive functor from $\mod \Lambda$ to $\mod \Gamma$.
\end{lemma}

Since $(Q,I)$ is a subquiver of $(Q^{Aus},I^{Aus})$, i.e. $\Lambda$ is a subalgebra of $\Gamma$, we get a restriction functor $\res:\mod \Gamma\rightarrow \mod\Lambda$.
Explicitly, for any $N=((N_i,N_\alpha)_{i\in Q_0,\alpha\in Q_1^{cyc}},(N_\beta)_{\beta \in Q_1^{Aus}})\in\mod \Gamma$, $\res(N)$ is defined as follows:

$\bullet$ For any $i\in Q_0$, $(\res(N))_i=N_i$;

$\bullet$ For any arrow $(\alpha:i\rightarrow j)\in Q_1$, if $\alpha\in Q_1^{ncyc}$, we set $(\res N)_\alpha=N_\alpha$; if $\alpha\in Q_1^{cyc}$, we set $(\res(N))_\alpha=N_{\alpha^-}N_{\alpha^+}$.

Since $\Lambda$ and $\Gamma$ are gentle algebras, their indecomposable modules are either string modules or band modules. We describe the action of $\Phi$ and $\res$ on string modules as follows.

$\bullet$ For a string $w=\alpha_1^{\epsilon_1}\alpha_2^{\epsilon_2}\dots\alpha_n^{\epsilon_n} \in\cs(\Lambda)$, denote its corresponding string module by $M(w)$. For $i=1,\dots,n$, if $\alpha_i\in Q_1^{cyc}$, we replace $\alpha_i$ by $\alpha_i^-\alpha_i^+$, and get a word in $\Gamma$, which is denoted by $\iota(w)$. Then it is easy to see that $\iota(w)\in \cs(\Gamma)$, we denote its string module by $N(\iota(w))$. Note that
$$\dimv N(\iota(w))=\dimv M(w)+\sum_{\begin{array}{cc}\alpha_i\in Q_1^{cyc}, \\ w=\alpha_1^{\epsilon_1}\alpha_2^{\epsilon_2}\dots\alpha_n^{\epsilon_n} \end{array}} \dimv S_{\alpha_i},$$
where $S_{\alpha_i}$ is the simple module corresponding to $\alpha_i\in Q_1^{cyc}\subseteq Q_0^{Aus}$.
In this way, we get a map $\iota:\cs(\Lambda)\rightarrow \cs(\Gamma)$, which is injective. It is easy to see that $\Phi(M(w))=N(\iota(w))$.

$\bullet$ For a string $v=\beta_1\beta_2\dots\beta_n \in\cs(\Gamma)$, denote its corresponding string module by $N(v)$. Obviously, $\res(N(v))$ is also a string module if $\res(N(v))\neq0$, we denote by $\pi^-(v)$ the string of $\res(N(v))$. Explicitly, we denote by $v'$ the longest substring of $v$ such that $s(v'),t(v')\in Q_0\subseteq Q_0^{Aus}$, then $\pi^-(v)$ is constructed from $v'$ by replacing $\alpha^-\alpha^+$ with $\alpha$ for each $\alpha\in Q_1^{cyc}$. Note that if $\res(N(v))=0$, then $\pi^-(v)$ is not defined. This only happens when $v=1_{(\alpha,i)}$ for some $\alpha\in Q_1^{cyc}\subseteq Q_0^{Aus}$.

Besides, there exists the shortest string $v''$ with $s(v''),t(v'')\in Q_0\subseteq Q_0^{Aus}$, such that $v$ is a substring of $v''$. Then $\pi^+(v)$ is constructed from $v''$ by replacing $\alpha^-\alpha^+$ with $\alpha$ for each $\alpha\in Q_1^{cyc}$. Obviously, $\pi^+(v)\in \cs(\Lambda)$, we denote its string module by $M(\pi^+(v))$.

In this way, we get two surjective maps $\pi^-,\pi^+:\cs(\Gamma)\rightarrow \cs(\Lambda)$, in fact, $\pi^-\iota=\Id=\pi^+\iota$.

\begin{example}
Keep the notations as in Example \ref{example} (a) and Example \ref{example 3} (a). Let $v=\alpha \gamma_2\beta$, which is a string in $\cs(\Lambda)$. Then
$\iota(v)=\alpha^{-1}\alpha^+\gamma_2\beta^{+}\beta^{-1}$, which is a string in $\cs(\Gamma)$.

For $\pi^+$ and $\pi^-$, we have $\pi^+(\alpha^+)=\alpha$, $\pi^+(\alpha^{-})=\alpha$, and $\pi^-(\alpha^+)= 1_{(1,1)}$, $\pi^-(\alpha^-)=1_{(2,1)}$.
Let $w=\alpha^{+}\gamma_2\beta^+$, which is a string in $\cs(\Gamma)$. Then $\pi^+(w)= \alpha \gamma_2\beta$, which is a string in $\cs(\Lambda)$, and $\pi^{-}(w)=\gamma_2$, which is a string in $\cs(\Lambda)$.
\end{example}
Note that $0\leq l(c)-l(\iota \pi^-(c))\leq 2$ for any string $c\in \cs(\Gamma)$ such that $\pi^-(c)$ is defined.

\begin{lemma}\label{lemma existence of band module}
Let $\Lambda=KQ/\langle I\rangle$ be a gentle algebra. Then $\Lambda$ admits band modules if and only if the Cohen-Macaulay Auslander algebra $\Aus(\Gproj\Lambda)$ of $\Lambda$ admits band modules.
\end{lemma}
\begin{proof}
Let $b=\alpha_1\alpha_2\cdots \alpha_{n-1}\alpha_n$ be a band in $\Lambda$. Then it is easy to see that $\iota(b)$ is also a band in $\Aus(\Gproj\Lambda)$.

Conversely, for any band $c$ in $\Aus(\Gproj\Lambda)$, if $s(c)=t(c)\in Q_0\subseteq Q_0^{Aus}$, it is easy to see that $\pi^-(c)$ is a band in $Q$. Otherwise, if $s(c)=t(c)\in Q_1^{cyc}$, then there exists $\alpha_1\in Q_1^{cyc}$ such that $s(c)=\alpha_1=t(c)$, which implies that $c$ is of form $\alpha_1^+ c_1 \alpha_1^-$ or $(\alpha_1^-)^{-1} c_1 (\alpha_1^+)^{-1}$, since there is only one arrow $\alpha_1^-$ starting from $\alpha_1$ and one arrow $\alpha_1^+$ ending to $\alpha_1$. We only check it for the first form since the second is similar. Then $d=c_1\alpha_1^{-}\alpha_1^+$ is also a band in $\Aus(\Gproj\Lambda)$. Since $s(d)=t(d)=s(\alpha_1)\in Q_0$, from the definition of $\pi^-$, we get that $s(\pi^-(d))=s(d)=t(d)=t(\pi^-(d))$. Together with
$\pi^-(d^m)=(\pi^-(d))^m$ for any $m>0$, it is easy to see that $\pi^-(d)$ is a band in $\Lambda$.
\end{proof}

\begin{theorem}\label{main theorem 2}
Let $\Lambda=KQ/\langle I\rangle$ be a gentle algebra. Then $\Lambda$ is representation-finite if and only if the Cohen-Macaulay Auslander algebra $\Gamma=\Aus(\Gproj\Lambda)$ of $\Lambda$ is representation-finite.
\end{theorem}
\begin{proof}
Theorem \ref{main theorem 1} shows that the Cohen-Macaulay Auslander algebra of $\Lambda$ is $KQ^{Aus}/\langle I^{Aus}\rangle$.

If $\Gamma=\Aus(\Gproj\Lambda)$ is representation-finite, then there is no band in $\Gamma$. Lemma \ref{lemma existence of band module} yields that there is no band in $\Lambda$. For each string $w=\alpha_1\alpha_2\cdots \alpha_n$ in $\cs(\Lambda)$, we have $\iota(w)\in\cs(\Gamma)$. Note that $\iota$ is injective. Since $\Gamma$ is representation-finite and every string defines a unique string module, there are only finitely many strings in $\Gamma$, which implies that there are only finitely many strings in $\Lambda$. Since $\Lambda$ admits no band module, we get that $\Lambda$ is representation-finite.

Conversely, if $\Lambda$ is representation-finite, then there is no band in $\Lambda$. Lemma \ref{lemma existence of band module} yields that there is no band in $\Gamma$. Let $c$ be a string in $\cs(\Lambda)$. For any string $v\in \cs(\Gamma)$ such that $\pi^-(v)=c$, it is easy to see that $\iota(c)$ is a substring of $v$ and $v$ is of form $\iota(c)$, $\alpha \iota(c)$, $\iota(c)\beta$ or $\alpha\iota(c)\beta$ for some $\alpha,\beta$ or their inverses in $(Q_1^{cyc})^\pm$.
Since $(Q_1^{cyc})^\pm$ is a finite set, there are only finitely many strings $v$ in $\cs(\Gamma)$ such that $\pi^-(v)=c$. Additionally, there are only finitely many strings in $\Lambda$, so there are only finitely many strings in $\Gamma$, and then $\Gamma=\Aus(\Gproj\Lambda)$ is representation-finite
since $\Gamma$ admits no band module.
\end{proof}

For a gentle algebra $\Lambda=KQ/\langle I\rangle$, if any indecomposable $\Lambda$-module $M$ is uniquely determined by its dimension vector, then there is no band module in $\Lambda$, since each band yields infinitely many indecomposable modules with the same dimension vector.

\begin{lemma}\label{lemma 2-cycles}
Let $\Lambda=KQ/\langle I\rangle$ be a gentle algebra such that there is no loop in $Q$. If any indecomposable $\Lambda$-module $M$ is uniquely determined by its dimension vector, then for any arrow $\alpha\in Q_1$, there is no arrow from $t(\alpha)$ to $s(\alpha)$, i.e., there is no oriented $2$-cycle in $Q$.
\end{lemma}
\begin{proof}
If there is an arrow $\beta:t(\alpha)\rightarrow s(\alpha)$, then there are two strings $s(\alpha)\xrightarrow{\alpha} t(\alpha)$, $ t(\alpha)\xrightarrow{\beta}s(\alpha)$. So there are two string modules with the same dimension vector, giving a contradiction.
\end{proof}

\begin{theorem}\label{main theorem 3}
Let $\Lambda=KQ/\langle I\rangle$ be a gentle algebra such that there is no loop in $Q$. If any indecomposable $\Lambda$-module $M$ is uniquely determined by its dimension vector, then any indecomposable $\Aus(\Gproj\Lambda)$-module $N$ is uniquely determined by its dimension vector.
\end{theorem}
\begin{proof}
If any indecomposable $\Lambda$-module $M$ is determined by its dimension vector, then there is no band in $\Lambda$ and Lemma \ref{lemma existence of band module} yields that $\Gamma=\Aus(\Gproj\Lambda)$ admits no band.
So there are only string modules in $\mod\Gamma$. We also get that any string in $\cs(\Lambda)$ is uniquely determined by its dimension vector up to the equivalence relation $\rho$.

For any vector $v=((v_i)_{i\in Q_0},(v_\alpha)_{\alpha\in Q_1^{cyc}})$ which is a dimension vector of a string $\Gamma$-module, set $v_1$ to be $(v_i)_{i\in Q_0}$ and $v_2$ to be $(v_\alpha)_{\alpha\in Q_1^{cyc}}$. If there are two strings $c,d\in\cs(\Gamma)$, such that $\dimv c=\dimv d=v$, then $l(c)=l(d)$.
If $v_1=0$, then $v$ is the dimension vector of a simple $\Gamma$-module $S_{\alpha}$ for some $\alpha\in Q_1^{cyc}\subseteq Q_0^{Aus}$, the result follows immediately since every simple module is uniquely determined by its dimension vector.

If $v_1\neq0$, then both $\pi^-(c)$ and $\pi^-(d)$ are well-defined, and $v_1$ is the dimension vector of the strings $\pi^-(c)$ and $\pi^-(d)$ in $\Lambda$. It follows that $\pi^-(c)\sim \pi^-(d)$ since $\dimv \pi^-(c)=\dimv \pi^-(d)$ and any string in $\Lambda$ is uniquely determined by its dimension vector up to the equivalence relation $\rho$. After choosing suitable representatives, we can assume that $\pi^-(c)=\pi^-(d)$.
We get that $\iota \pi^-(c)=\iota \pi^-(d)$ appears as substrings of $c$ and $d$. Recall that $0\leq l(c)-l(\iota \pi^-(c))\leq 2$.

{\bf Case (1).} If $l(c)=l(\iota \pi^-(c))$, then $c=\iota\pi^-(c)$, which also implies $d=\iota \pi^-(d)$ by $l(c)=l(d)$. Then $c=d$ since $\pi^-(c)=\pi^-(d)$ and $\iota$ is injective.

{\bf Case (2).} $l(c)-l(\iota \pi^-(c))=1$. We assume that $\iota\pi^-(c)=\iota\pi^-(d)$ is
\[\xymatrix{  b_1 \ar@{-}[r]^{\alpha_1} &b_2 \ar@{-}[r]^{\alpha_2}& \cdots \ar@{-}[r]^{\alpha_{n-1}}& b_n\ar@{-}[r]^{\alpha_n} &b_{n+1}.}\]
Suppose for a contradiction that $c$ is not equivalent to $d$.

Since $\dimv c=\dimv d$, there exists some $\alpha\in Q_1^{cyc}$ such that $c$ and $d$ are of the following forms:
\[\xymatrix{ c_1:\quad \alpha& b_1 \ar[l]_{\quad\quad\alpha^+} \ar@{-}[r]^{\alpha_1} &b_2 \ar@{-}[r]^{\alpha_2}& \cdots \ar@{-}[r]^{\alpha_{n-1}}& b_n\ar@{-}[r]^{\alpha_n} &b_{n+1},}\]
\[\xymatrix{ c_2:\quad \alpha \ar[r]^{\quad\quad\alpha^-}& b_1  \ar@{-}[r]^{\alpha_1} &b_2 \ar@{-}[r]^{\alpha_2}& \cdots \ar@{-}[r]^{\alpha_{n-1}}& b_n\ar@{-}[r]^{\alpha_n} &b_{n+1},}\]
\[\xymatrix{ c_3:\quad  b_1  \ar@{-}[r]^{\quad\quad\alpha_1} &b_2 \ar@{-}[r]^{\alpha_2}& \cdots \ar@{-}[r]^{\alpha_{n-1}}& b_n\ar@{-}[r]^{\alpha_n} &b_{n+1} \ar[r]^{\quad\alpha^+}&\alpha,}\]
\[\xymatrix{ c_4:\quad  b_1  \ar@{-}[r]^{\quad\quad\alpha_1} &b_2 \ar@{-}[r]^{\alpha_2}& \cdots \ar@{-}[r]^{\alpha_{n-1}}& b_n\ar@{-}[r]^{\alpha_n} &b_{n+1} &\alpha \ar[l]_{\quad\alpha^-}.}\]
If $c=c_1$, then $d$ can only be of form $c_3$ or $c_4$ since there is no loop in $Q$.
First, if $d=c_3$, then $\pi^+(d)= \pi^+(\iota \pi^-(d))\alpha^{-1}= \pi^+(\iota \pi^-(c))\alpha^{-1}$, and $\pi^+(c)= \alpha\pi^+(\iota \pi^-(c))$. Then $\dimv \pi^+(d)=\dimv \pi^+(c)$, which means that $\pi^+(d)\sim\pi^+(c)$. If $\pi^+(d)=\pi^+(c)$, then from the definition of $\pi^+$, we get that  $\alpha^-=\alpha_1$, $\alpha^+=\alpha_2$, $\alpha_1=\alpha_3$ and so on. So $\alpha_3=\alpha^-$, which yields that $\alpha^-\alpha^+\alpha^-$ is a string. However, $t(\alpha)=t(\alpha^-)=s(\alpha^+)=s(\alpha)$, which means that $\alpha$ is a loop in $Q$, contradicts to the assumption of $Q$. If $\pi^+(d)=(\pi^+(c))^{-1}$, then $\pi^+(\iota \pi^-(c))\alpha^{-1}=  (\alpha\pi^+(\iota \pi^-(c)))^{-1}= (\pi^+(\iota \pi^-(c)))^{-1}\alpha^{-1}$, which means that $\pi^+(\iota \pi^-(c))=(\pi^+(\iota \pi^-(c)))^{-1}$, giving a contradiction to Remark \ref{lemma string not equal to its inverse}.

Second, if $d=c_4$, then
\[\xymatrix{ \iota\pi^+(c):\quad b_{n+1} & \alpha\ar[l]_{\quad\quad\quad\quad\alpha^-}& b_1 \ar[l]_{\quad\alpha^+} \ar@{-}[r]^{\alpha_1} &b_2 \ar@{-}[r]^{\alpha_2}& \cdots \ar@{-}[r]^{\alpha_{n-1}}& b_n\ar@{-}[r]^{\alpha_n} &b_{n+1}}\]
is a string, and its starting point and ending point coincide. From $\iota\pi^+(d)$, it is easy to see that $(\iota\pi^+(c))^m$ is also a string for any $m>0$, which implies that there is a band in $\Gamma$, giving a contradiction.
In conclusion, $d=c$ if $c$ is of form $c_1$.

For $c$ is one of forms $c_2, c_3$ and $c_4$, the proof is similar to the above, we omit the proof here.

{\bf Case (3).} $l(c)-l(\iota \pi^-(c))=2$. We assume that
\[\xymatrix{ \iota\pi^-(c)=\iota\pi^-(d):&b_1 \ar@{-}[r]^{\alpha_1} &b_2 \ar@{-}[r]^{\alpha_2}& \cdots \ar@{-}[r]^{\alpha_{n-1}}& b_n\ar@{-}[r]^{\alpha_n} &b_{n+1}.}\]
There are four cases for the structure of $c$.

{\bf Case (3a).} $c$ is
\[\xymatrix{ c:\quad \alpha& b_1 \ar[l]_{\quad\quad\alpha^+} \ar@{-}[r]^{\alpha_1} &b_2 \ar@{-}[r]^{\alpha_2}& \cdots \ar@{-}[r]^{\alpha_{n-1}}& b_n\ar@{-}[r]^{\alpha_n} &b_{n+1}\ar[r]^{\quad\beta^+}&\beta}\]
for some $\alpha,\beta\in Q_1^{cyc}$.
If $\alpha=\beta$, then $d=c$ since $\dimv c=\dimv d$ and $Q$ has no loop.

For $\alpha\neq\beta$, suppose for a contradiction that $d$ is not equivalent to $c$.
Then $d$ is one of the following forms:
\[\xymatrix{ d_1:\quad \beta& b_1 \ar[l]_{\quad\quad\beta^+} \ar@{-}[r]^{\alpha_1} &b_2 \ar@{-}[r]^{\alpha_2}& \cdots \ar@{-}[r]^{\alpha_{n-1}}& b_n\ar@{-}[r]^{\alpha_n} &b_{n+1}\ar[r]^{\quad\alpha^+}&\alpha ,}\]
\[\xymatrix{ d_2:\quad \beta& b_1 \ar[l]_{\quad\quad\beta^+} \ar@{-}[r]^{\alpha_1} &b_2 \ar@{-}[r]^{\alpha_2}& \cdots \ar@{-}[r]^{\alpha_{n-1}}& b_n\ar@{-}[r]^{\alpha_n} &b_{n+1}&\alpha \ar[l]_{\quad\alpha^-},}\]
\[\xymatrix{ d_3:\quad \beta \ar[r]^{\quad\quad\beta^-}& b_1  \ar@{-}[r]^{\alpha_1} &b_2 \ar@{-}[r]^{\alpha_2}& \cdots \ar@{-}[r]^{\alpha_{n-1}}& b_n\ar@{-}[r]^{\alpha_n} &b_{n+1}\ar[r]^{\quad\alpha^+}&\alpha ,}\]
\[\xymatrix{ d_4:\quad \beta \ar[r]^{\quad\quad\beta^-}&b_1 \ar@{-}[r]^{\alpha_1} &b_2 \ar@{-}[r]^{\alpha_2}& \cdots \ar@{-}[r]^{\alpha_{n-1}}& b_n\ar@{-}[r]^{\alpha_n} &b_{n+1}&\alpha \ar[l]_{\quad\alpha^-}.}\]

For $d=d_1$, if $n=0$, then $d=c^{-1}$, a contradiction.
If $n>0$, then there are two arrows $\alpha^+,\beta^+$ from $b_1$, and $\alpha_1$ is of form $\alpha_1:b_2\rightarrow b_1$ since $\Gamma$ is gentle. Then $\beta^+\alpha_1,\alpha^+\alpha_1\notin I^{Aus}$, a contradiction.
For $d=d_2$, if $n=0$, then there is an oriented $2$-cycle $b_1\xrightarrow{\alpha^+}\alpha\xrightarrow{\alpha^-}b_1$ in $\Gamma$, a contradiction; if $n>0$, then similar to the above case $d=d_1$, we can get that it is also impossible.
For $d=d_3$, it is easy to see that $b_{n+1}=b_1$, then there is an oriented $2$-cycle $b_1\xrightarrow{\beta^+}\beta\xrightarrow{\beta^-}b_1$, a contradiction.
For $d=d_4$, there is an oriented $2$-cycle $b_{n+1}\xrightarrow{\beta} b_1\xrightarrow{\alpha}b_{n+1}$ in $Q$, a contradiction to Lemma \ref{lemma 2-cycles}.
Therefore, $d$ is equivalent to $c$ in this case.

{\bf Case (3b).}
$c$ is of form
\[\xymatrix{ c:\quad \alpha& b_1 \ar[l]_{\quad\quad\alpha^+} \ar@{-}[r]^{\alpha_1} &b_2 \ar@{-}[r]^{\alpha_2}& \cdots \ar@{-}[r]^{\alpha_{n-1}}& b_n\ar@{-}[r]^{\alpha_n} &b_{n+1}&\beta\ar[l]_{\quad\beta^-}}\]
for some $\alpha,\beta\in Q_1^{cyc}$.
If $\alpha=\beta$, then $d=c$ since $\dimv c=\dimv d$ and $Q$ has no loop.
For $\alpha\neq\beta$, suppose for a contradiction that $d$ is not equivalent to $c$.
Then $d$ is also one of the forms $d_1,d_2,d_3,d_4$ as described in Case (3a).

For $d=d_1$, if $n=0$, then $b_1\xrightarrow{\beta^+} \beta\xrightarrow{\beta^-}b_1$ is an oriented $2$-cycle, a contradiction.
If $n>0$, then we can check that it is impossible similar to Case (3a).
For $d=d_2$, if $n=0$, then $b_1\xrightarrow{\beta^+} \beta\xrightarrow{\beta^-}b_1$ is an oriented $2$-cycle, a contradiction.
If $n>0$, then there are two arrows $\alpha^-,\beta^-$ ending to $b_{n+1}$, and $\alpha_n$ is of form $\alpha_n:b_{n+1}\rightarrow b_{n}$ since $\Gamma$ is gentle. Then $\alpha_n\beta^-,\alpha_n\alpha^-\notin I^{Aus}$, a contradiction.
For $d=d_3$, it is easy to see that $\dimv\pi^+(d)=\dimv\pi^+(c)$, so $\pi^+(d)\sim\pi^+(c)$ and then $\iota\pi^+(d)\sim\iota \pi^+(c)$,
that is 
\[\xymatrix{ \iota\pi^+(c):\quad t(\alpha) & \alpha\ar[l]_{\quad\quad\quad\alpha^-}& b_1 \ar[l]_{\alpha^+} \ar@{-}[r]^{\alpha_1} &b_2 \ar@{-}[r]^{\alpha_2}& \cdots \ar@{-}[r]^{\alpha_{n-1}}& b_n\ar@{-}[r]^{\alpha_n} &b_{n+1}&\beta\ar[l]_{\quad\beta^-}&s(\beta)\ar[l]_{\beta^+}}\]
and
\[\xymatrix{ \iota\pi^+(d):\quad  s(\beta)\ar[r]^{\quad\quad\beta^+}& \beta\ar[r]^{\beta^-} &b_1 \ar@{-}[r]^{\alpha_1} &b_2 \ar@{-}[r]^{\alpha_2}& \cdots \ar@{-}[r]^{\alpha_{n-1}}& b_n\ar@{-}[r]^{\alpha_n} &b_{n+1} \ar[r]^{\quad\alpha^+}&\alpha \ar[r]^{\alpha^-}&t(\alpha)}\]
are equivalent under $\rho$, which implies that $\iota\pi^+(c)=(\iota\pi^+(d))^{-1}$. Then $(\iota\pi^{-}(c))=(\iota\pi^-(c))^{-1}$, which is impossible.
For $d=d_4$, obviously, $b_{n+1}=b_1$ and so $b_1\xrightarrow{\alpha^+}\alpha\xrightarrow{\alpha^-}b_1$ is an oriented $2$-cycle, a contradiction.
Therefore, in this case, $d$ is equivalent to $c$.

{\bf Case (3c).} $c$ is of form
\[\xymatrix{ c:\quad \alpha \ar[r]^{\quad\quad\alpha^-}& b_1  \ar@{-}[r]^{\alpha_1} &b_2 \ar@{-}[r]^{\alpha_2}& \cdots \ar@{-}[r]^{\alpha_{n-1}}& b_n\ar@{-}[r]^{\alpha_n} &b_{n+1}\ar[r]^{\quad\beta^+}&\beta}\]
for some $\alpha,\beta\in Q_1^{cyc}$.
This case is similar to Case (3b), we omit the proof here.

{\bf Case (3d).} $c$ is
\[\xymatrix{ c:\quad \alpha \ar[r]^{\quad\quad\alpha^-}& b_1  \ar@{-}[r]^{\alpha_1} &b_2 \ar@{-}[r]^{\alpha_2}& \cdots \ar@{-}[r]^{\alpha_{n-1}}& b_n\ar@{-}[r]^{\alpha_n} &b_{n+1}&\beta \ar[l]_{\quad\beta^-}}\]
for some $\alpha,\beta\in Q_1^{cyc}$.
This case is similar to Case (3a), we omit the proof here.

To sum up, when $l(c)-l(\iota \pi^-(c))=2$, we get that $c$ is equivalent to $d$.

Therefore, for any strings $c,d$ in $\cs(\Gamma)$, if $\dimv c=\dimv d$, then $c\sim d$. For any indecomposable $\Gamma$-module $N$, we get that $N$ is a string module, which is uniquely determined by its string up to the equivalent relation $\rho$, and so $N$ is uniquely determined by its dimension vector.
\end{proof}

The following example shows that the converse of Theorem \ref{main theorem 3} is not valid.
\begin{example}
Let $\Lambda =KQ/\langle I\rangle$ be a gentle algebra with \[\xymatrix{ Q:& 1\ar@<0.5ex>[r]^{\alpha} & 2\ar@<0.5ex>[l]^{\beta}  &I=\{\alpha\beta,\beta\alpha\}.}\]
Then $Q^{Aus}$ is as following diagram shows and $I^{Aus}=\{\beta^+\alpha^-,\alpha^+\beta^- \}$.
\setlength{\unitlength}{1mm}
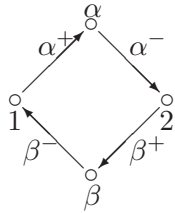
\begin{center}
\begin{picture}(40,25)

\put(10,10){\circle{1.3}}
\put(30,10){\circle{1.3}}
\put(20,20){\circle{1.3}}
\put(20,0){\circle{1.3}}
\put(11,11){\vector(1,1){8}}
\put(21,19){\vector(1,-1){8}}
\put(29,9){\vector(-1,-1){8}}
\put(19,1){\vector(-1,1){8}}
\put(9,6){$1$}
\put(19,21){$\alpha$}
\put(29,6){$2$}
\put(19,-4){$\beta$}
\put(13,16){$\alpha^+$}
\put(25,16){$\alpha^-$}
\put(25,2){$\beta^+$}
\put(11,2){$\beta^-$}

\end{picture}
\vspace{0.5cm}
\end{center}

It is easy to see that any indecomposable $KQ^{Aus}/\langle I^{Aus}\rangle$-module is uniquely determined by its dimension vector. However, the indecomposable projective $\Lambda$-modules $P_1,P_2$ corresponding to vertices $1,2$ respectively, have the same dimension vector.
\end{example}

\begin{remark}\label{remark quiver of gentle algebra}
Let $\Lambda=KQ/\langle I\rangle$ be a gentle algebra. If any indecomposable $\Lambda$-module $M$ is uniquely determined by its dimension vector, then for any loop $\alpha:i\rightarrow i$ with $i$ a vertex, there is no arrow $\beta\neq \alpha$ starting from $i$ or ending to $i$.
\end{remark}
\begin{proof}
Since $\Lambda$ is a gentle algebra, for any loop $\alpha:i\rightarrow i$, we have $\alpha^2\in I$.
First, note that there is not another loop $\beta$ with the same starting point $i$. Otherwise, we also have $\beta^2\in I$. Then $\beta\alpha,\alpha\beta\notin I$ since $\Lambda$ is gentle, contradicts to the fact $\Lambda$ is finite-dimensional.

If there is another arrow $\beta:i\rightarrow j$, then $j\neq i$. Obviously, $\beta\alpha\notin I$. So there are two nonequivalent strings
$i\xrightarrow{\alpha} i\xrightarrow{\beta}j$ and $i\xleftarrow{\alpha}i\xrightarrow{\beta}j$, which have the same dimension vector, a contradiction.

If there is another arrow $\beta:j\rightarrow i$, it is similar to the above case, we omit the proof here.
\end{proof}

\begin{example}
Let $\Lambda=KQ/\langle I\rangle$ be a gentle algebra with $Q_0=\{1\}$, $Q_1=\{\alpha:1\rightarrow 1\}$. Then $I=\{\alpha^2\}$.
Let $KQ^{Aus}/\langle I^{Aus}\rangle$ be the Cohen-Macaulay Auslander algebra of $\Lambda$. Then $Q^{Aus}$ is as the following diagram shows and $I^{Aus}=\{\alpha^+\alpha^-\}$.
\[\xymatrix{ Q^{Aus}:&1\ar@<0.5ex>[r]^{\alpha^+} & 2\ar@<0.5ex>[l]^{\alpha^-}  }\]
It is easy to that $KQ^{Aus}/\langle I^{Aus}\rangle$ does not satisfy that any indecomposable module is uniquely determined by its dimension vector.
\end{example}

\begin{corollary}
Let $\Lambda=KQ/\langle I\rangle$ be a gentle algebra with $Q$ connected. Assume that $\Lambda$ satisfies that any indecomposable $\Lambda$-module $M$ is uniquely determined by its dimension vector. If there are two indecomposable $\Aus(\Gproj(\Lambda))$-modules with the same dimension vector, then $\Lambda$ is isomorphic to the local ring $K[X]/\langle X^2\rangle$.
\end{corollary}
\begin{proof}
Since any indecomposable $\Lambda$-module $M$ is uniquely determined by its dimension vector, if there is no loop in $Q$, Theorem \ref{main theorem 3} yields that any indecomposable $\Aus(\Gproj\Lambda)$-module $N$ is determined by its dimension vector, a contradiction. So there is at least one loop in $Q$. Furthermore, Remark \ref{remark quiver of gentle algebra} implies that $Q_0=\{v\}$, $Q_1=\{\alpha:v\rightarrow v\}$ since $Q$ is connected, and so $\Lambda\cong K[X]/\langle X^2\rangle$.
\end{proof}

At the end of this section, we give the following proposition for schurian gentle algebras.
Recall that an algebra $A=KQ/I$ is \emph{schurian} if $\dim_k\Hom_A(P_i,P_j)\leq1$ for any two vertices $i,j$ of $Q$, or in other words, the entries of its Cartan matrix
are only $0$ or $1$.

\begin{proposition}\label{proposition auslander algebras are schurian}
Let $\Lambda=KQ/\langle I\rangle$ be a schurian gentle algebra. Then its Cohen-Macaulay Auslander algebra $\Gamma=\Aus(\Gproj\Lambda)$ is also a schurian gentle algebra.
\end{proposition}
\begin{proof}
Let $P$ be an indecomposable projective $\Gamma$-module corresponding to some vertex $b_1\in Q^{Aus}_0$. Since $\Gamma$ is a gentle algebra, $P$ is a string module, see e.g. \cite[Section 4]{Ka}. Denote by $w$ its string. Then from \cite[Section 4]{Ka}, we get that $w$ is of form
\[\xymatrix{ w:\quad  b_{n+m+1}&  \cdots \ar[l]_{\quad\quad\quad\beta_{m}}& b_{n+2} \ar[l]_{\beta_2}& b_1\ar[l]_{\quad\beta_1} \ar[r]^{\alpha_1} &b_2 \ar[r]^{\alpha_2}& \cdots \ar[r]^{\alpha_{n-1}}& b_n\ar[r]^{\alpha_n\quad} &b_{n+1},}\]
or
\[\xymatrix{ w:\quad a_1 \ar[r]^{\quad\gamma_1} &a_2 \ar[r]^{\gamma_2}& \cdots \ar[r]^{\gamma_{l-1}}& a_l\ar[r]^{\gamma_l\quad} &a_{l+1},}\]
where the paths $\beta_m\dots \beta_2\beta_1$, $\alpha_n\dots\alpha_2\alpha_1$ and $\gamma_l\dots\gamma_2\gamma_1$
appearing above are maximal, e.g. there does not exist $\beta\in Q_1^{Aus}$ such that $\beta\beta_m\notin I^{Aus}$, see e.g. \cite{ASS,Ka}.
Therefore, we only need to check that the string $w$ passes through any vertex at most once.

For $w$ is of the first case, we claim that $b_1,b_{n+1},b_{n+m+1}\in Q_0\subseteq Q_0^{Aus}$. In fact, if $b_1\notin Q_0$, then $b_1=\alpha$ for some $\alpha\in Q_1^{cyc}\subseteq Q_0^{Aus}$. Then there are two arrows $\alpha_1,\beta_1$ starting from $\alpha$. Recall that there is only one arrow $\alpha^-$ starting from $\alpha$ in $Q^{Aus}$, a contradiction. If $b_{n+1}\notin Q_0$, then $b_{n+1}=\beta$ for some $\beta\in Q_1^{cyc}\subseteq Q_0^{Aus}$.
Since there is only one arrow $\beta^+$ ending to $\beta$ in $\Q^{Aus}$, $\alpha_n=\beta^+$. However, $\beta^-\beta^+\notin I^{Aus}$, so we get that $\alpha_n\cdots \alpha_2\alpha_1$ is not maximal, a contradiction. For $b_{n+m+1}\in Q_0$, it is similar to the above.

 It is easy to see that $\pi^-(w)\in\cs(\Lambda)$ is the string of the indecomposable projective $\Lambda$-module corresponding to the vertex $b_1\in Q_0$. From $\Lambda$ is schurian, we get that $\pi^-(w)$ does not pass through any vertex more than once.
It follows that $w$ does not pass through any vertex in $Q_0\subseteq Q_0^{Aus}$ more than once. Furthermore, if $w$ passes through a vertex $\alpha\in Q_1^{cyc}\subseteq Q_0^{Aus}$ at least twice, then
$w$ must pass through $s(\alpha)$ or $t(\alpha)$ at least twice, which yields that $\pi^-(w)$ passes through $s(\alpha)$ or $t(\alpha)$ at least twice, a contradiction.

If $w$ is of the second case, similar to the first case, we get that $a_{l+1}\in Q_0\subseteq Q_0^{Aus}$. If $a_1\in Q_0$, then it is similar to the first case. If $a_1=\alpha\in Q_1^{cyc}$, then $\alpha_1=\alpha^-$ since there is only one arrow $\alpha^-$ starting from $\alpha$. It is easy to see that $\pi^+(w)\in\cs(\Lambda)$ is the string of a quotient of the indecomposable projective $\Lambda$-module $_\Lambda P_{s(\alpha)}$ corresponding to the vertex $s(\alpha)$. Let $v$ be the string of $_\Lambda P_{s(\alpha)}$. From the above, we know that $v$ does not pass through any vertex more than once. Note that $w$ is a substring of $v$, so $w$ does not pass through any vertex more that once.

Therefore,  $\Gamma$ is a schurian algebra.
\end{proof}

\end{document}